\documentclass[12pt, a4paper]{amsart}
\usepackage{hyperref, fullpage, amsmath, amsfonts, amsthm, amssymb, stmaryrd, paralist}

\newcommand{\length}{\operatorname{length}}

\newcommand{\rk}{\operatorname{rk}}

\newcommand{\Z}{\mathbb Z}
\newcommand{\N}{\mathbb N}

\newcommand{\normal}{\triangleleft}
\newcommand{\normaleq}{\trianglelefteq}

\renewcommand{\epsilon}{\varepsilon}

\theoremstyle{plain}
\newtheorem{Theorem}{Theorem}[section]
\newtheorem{Lemma}[Theorem]{Lemma}

\newtheorem{Proposition}[Theorem]{Proposition}
\newtheorem{Corollary}[Theorem]{Corollary}
\newtheorem{Question}[Theorem]{Question}

\theoremstyle{definition}

\theoremstyle{remark}
\newtheorem{Remark}[Theorem]{Remark}
\newtheorem{Example}[Theorem]{Example}

\title{On small profinite groups}

\author{Patrick Helbig}

\date{}

\begin{document}

\begin{abstract}
A profinite group is called \textit{small} if it has only finitely many open subgroups of index $n$ for each positive integer $n$. We show that every Frattini cover of a small profinite group is small. A profinite group is called \textit{strongly complete} if every subgroup of finite index is open. We show that two profinite groups that are elementarily equivalent, in the first-order language of groups, are isomorphic if one of them is strongly complete, extending a result of Moshe Jarden and Alexander Lubotzky which treats the case of finitely generated profinite groups.
\end{abstract}

\maketitle

\section{Introduction}

A profinite group is called \textit{small} if it has only finitely many open subgroups of index $n$ for each positive integer $n$. Motivation for considering small profinite groups comes from the fact that the absolute Galois group of a field is small if and only if it is determined up to an isomorphism by the field's theory in the language of rings, cf. \cite[Proposition 20.4.6]{FJ}. Related to this fact is the result that a small profinite group is determined up to an isomorphism by the set of those finite groups which occur as quotients by open normal subgroups (in contrast to the general case, where a profinite group is determined by its continuous finite quotients and the quotient maps between them), see \cite[Proposition 16.10.7]{FJ}.

In the first part of this note we characterize small profinite groups as follows:
\begin{Proposition} A profinite group $G$ is small if and only if $r_S(H) < \infty$ for all open subgroups $H$ of $G$ and all finite simple groups $S$.
\end{Proposition}
Here, $r_S(H)$ denotes the $S$-rank of $H$, see Remark \ref{RemarkFrattiniCover} (iii). Adapting the proof of \cite[Lemma 4.2]{FeJ}, where a similar condition was established in the case where $S$ is abelian, we use this characterization to show:
\begin{Theorem} Every Frattini cover of a small profinite group is small.
\end{Theorem}
Since the universal Frattini cover of a profinite group is a projective profinite group, this result can be used, for example, to find small profinite groups which can be realized as absolute Galois groups, see \cite[Corollary 23.1.3]{FJ}.

In the second part we consider the theory of profinite groups in the first-order language of groups, and we say that $G$ and $H$ are \textit{elementarily equivalent}, denoted $G \equiv H$, if $G$ and $H$ satisfy the same sentences in the language of groups. We generalize \cite[Theorem 1.7]{JL}, which states:
\begin{Theorem} \label{IntroductionFG} (Jarden, Lubotzky) Let $G$ and $H$ be profinite groups with $G \equiv H$. If one of them is finitely generated, then $G \cong H$.
\end{Theorem}
It is an immediate consequence of \cite[Proposition 2.20]{F} that the same holds if one of the groups is small and abelian instead of finitely generated, which will also follow from our result. The proof of Theorem \ref{IntroductionFG} uses two facts: first, that finitely generated profinite groups are small, and second, the existence of certain group words whose corresponding verbal subgroups are open in the finitely generated profinite group. The second fact is a deep result proved in \cite{NS}, where it is used to show that all finitely generated profinite groups are \textit{strongly complete}, i.e. every subgroup of finite index is open, see \cite[Theorems 1.1 and 1.2]{NS}. We begin by formulating the approach used in \cite{JL} in a slightly more general setting that takes the above points into consideration:
\begin{Proposition}\label{IntroductionIC} Let $G$ and $H$ be profinite groups with $G \equiv H$. Suppose that $G$ is small and that for each finite group $A$ there exists a group word $w$ with $w(A) = \{1\}$ such that the verbal subgroup $w(G)$ is open in $G$. Then $G \cong H$.
\end{Proposition}
Then we use the characterization of strong completeness proved in \cite[Theorem 2]{SW}, involving varieties generated by finite groups, and the theorem of Oates and Powell (cf. \cite[Corollary 52.12]{Nm}) that the laws of a finite group have a finite basis, to show that strongly complete profinite groups satisfy the assumptions of Proposition \ref{IntroductionIC}, therefore showing:
\begin{Theorem} \label{IntroductionSC} Let $G$ and $H$ be profinite groups with $G \equiv H$. If one of them is strongly complete, then $G \cong H$.
\end{Theorem}
Since finitely generated profinite groups are strongly complete by \cite[Theorem 1.1]{NS}, Theorem \ref{IntroductionSC} implies Theorem \ref{IntroductionFG}. In the third section we will give an example of a strongly complete profinite group which is neither finitely generated nor abelian, showing that Theorem \ref{IntroductionSC} is more general than both Theorem \ref{IntroductionFG} and the result concerning small abelian profinite groups from \cite{F}.

\section{Frattini covers of small profinite groups}

\begin{Remark} \label{RemarkFrattiniCover} Let $G$ be a profinite group. \begin{enumerate}[(i)]

\item By $\rk(G)$ we denote the \textit{rank} of $G$, see \cite[Chapter 17.1]{FJ}.

\item The \textit{Frattini subgroup} of $G$, denoted $\Phi(G)$, is the intersection of all maximal, open subgroups of $G$, see \cite[Section 22.1]{FJ}. A \textit{Frattini cover} of $G$ is an epimorphism (i.e. a continuous surjective group homomorphism) $\varphi : K \to G$ from a profinite group $K$ to $G$ with $\ker(\varphi) \subseteq \Phi(G)$, see \cite[Section 22.5]{FJ}. Every Frattini cover $\varphi : K \to G$ has the following property: If $\Delta$ is a continuous quotient of $K$, then there exist a continuous quotient $\Gamma$ of $G$ with $\rk(\Delta) = \rk(\Gamma)$ and an epimorphism $\Delta \to \Gamma$. (If $K = \tilde G$ is the universal Frattini cover of $G$, then this follows from \cite[Corollary 22.5.3, Lemma 22.6.3]{FJ}; the general case follows because $\tilde G \to G$ factors through $K \to G$, see \cite[Proposition 22.6.1]{FJ}.)

\item Let $S$ be a finite simple group. Let $M_S(G)$ denote the intersection of all open normal subgroups $N$ of $G$ with $G/N \cong S$. Then $G/M_S(G) \cong S^\lambda$ for a unique cardinal number $\lambda$ which is called the \textit{$S$-rank} of $G$ and is denoted by $r_S(G)$, see \cite[Section 8.2]{RZ}.
\end{enumerate}
\end{Remark}

We first establish the characterization of small profinite groups announced in the introduction.

\begin{Proposition} \label{SRankSmall} A profinite group $G$ is small if and only if $r_S(H) < \infty$ for all open subgroups $H$ of $G$ and all finite simple groups $S$.
\end{Proposition}
\begin{proof} Suppose that $G$ is small. Then every open subgroup of $G$ is small, hence it suffices to show that $r_S(G) < \infty$ for all finite simple groups $S$. Since $G$ is small, $M_S(G)$ is an intersection of finitely many open subgroups of $G$, hence open in $G$ and thus of finite index. Then $G/M_S(G)$ is a finite group and so $r_S(G)$ is finite.

Conversely, suppose that $r_S(H) < \infty$ for all open subgroups $H$ of $G$ and all finite simple groups $S$. Observe that it suffices to show that for any finite group $F$ there are only finitely many open normal subgroups $N$ of $G$ with $G/N \cong F$.

Let $F$ be a finite group and consider a composition series
$$\{1\} = F_0 \normal F_1 \normal ... \normal F_n = F$$
of $F$. We claim that, for all $k \in \{0, ..., n\}$ and all open subgroups $H$ of $G$, there exist only finitely many open normal subgroups $N$ of $H$ such that $H/N \cong F_k$. The case $H = G$ and $k = n$ then finishes the proof. We proceed by induction on $k$.

The claim is true if $k = 0$. Let $0 < k \leq n$ and let $H$ be an open subgroup of $G$. Let $S := F_k / F_{k-1}$. By assumption, $r_S(H)$ is finite, hence there are only finitely many open normal subgroups $K_1, ..., K_m$ of $H$ with $H/K_i \cong S$ for each $i$. Moreover, for each $i$ the group $K_i$ is an open subgroup of $G$, so by induction there are only finitely many open normal subgroups $N_{i1}, ..., N_{in_i}$ of $K_i$ with $K_i/N_{ij} \cong F_{k-1}$ for each $j$.

Now let $N$ be an open normal subgroup of $H$ such that $H/N \cong F_k$. We show that $N = N_{ij}$ for some $i \in \{1, ..., m\}$ and $j \in \{1, ..., n_i\}$, thus showing that there are only finitely many such $N$. Let $\pi : H \to H/N$ be the quotient map and $\varphi : H/N \to F_k$ an isomorphism. Let $K := \pi^{-1}(\varphi^{-1}(F_{k-1}))$. Then $N \normaleq K \normaleq H$ are open normal subgroups of each other with $K/N \cong F_{k-1}$ and
$$H/K \cong \frac{H/N}{K/N} \cong F_k / F_{k-1} = S,$$
hence $K = K_i$ for some $i \in \{1, ..., m\}$ and so $N = N_{ij}$ for some $j \in \{1, ..., n_i\}$.
\end{proof}

Now we work towards applying this characterization to Frattini covers of small profinite groups.

\begin{Lemma} \label{SRankEqualsRank} Let $G$ be an infinite profinite group, $S$ a finite simple group, $\lambda$ a cardinal number, and $\varphi : S^\lambda \to G$ an epimorphism. Then $\rk(G) = r_S(G)$.
\end{Lemma}
\begin{proof} The kernel of $\varphi$ is a closed normal subgroup of $S^\lambda$, thus there exists a closed normal subgroup $U \cong S^\kappa$ of $S^\lambda$, where $\kappa$ is a cardinal number, with $S^\lambda = \ker(\varphi) \times U \cong \ker(\varphi) \times S^\kappa$; see \cite[Lemmas 22.7.2, 22.7.3]{FJ} for the case of abelian $S$ and \cite[Lemma 8.2.4]{RZ} for the case of non-abelian $S$. It follows that $G \cong S^\kappa$. Since $G$ is infinite, also $\kappa$ is, thus $\rk(G) = \rk(S^\kappa) = \kappa = r_S(S^\kappa) = r_S(G)$.
\end{proof}

We now establish above condition for a Frattini cover of a small profinite group. To do so we adapt the proof of \cite[Lemma 4.2]{FeJ}, where a similar result is proved in the case of abelian $S$.

\begin{Lemma} \label{SRankFinite} Let $G$ be a small profinite group and $\varphi : K \to G$ a Frattini cover of $G$. Then one has $r_S(H) < \infty$ for all open subgroups $H$ of $K$ and all finite simple groups $S$.
\end{Lemma}
\begin{proof} Let $H$ be an open subgroup of $K$ and $S$ a finite simple group. Then $H$ contains an open normal subgroup $H_0$ of $K$. By \cite[Lemma 8.2.5 (d)]{RZ} we have
$$r_S(H) \leq r_S(H_0) + r_S(H/H_0).$$
Since $H/H_0$ is a finite group, it has finite $S$-rank, hence it suffices to show that $r_S(H_0)$ is finite. Thus, we can assume that $H$ is normal in $K$.

Let $N := M_S(H)$. Note that $N$ is invariant under all continuous automorphisms of $H$, hence $N$ is a normal subgroup of $K$. Let $\Delta := K / N$ and $\Delta_0 := H/N$. Since $\Delta$ is a continuous quotient of $K$, there exist a continuous quotient $\Gamma$ of $G$ with $\rk(\Delta) = \rk(\Gamma)$ and an epimorphism $\phi : \Delta \to \Gamma$. Let $\pi : G \to \Gamma$ be the quotient map, $\Gamma_0 := \phi(\Delta_0)$, and $G_0 := \pi^{-1}(\Gamma_0)$.
Observe that $(\Gamma : \Gamma_0) \leq (\Delta : \Delta_0) = (K : H) < \infty$, since $\phi$ is surjective.

Now assume that $\lambda := r_S(H)$ is infinite. Then
$$\rk(\Delta_0) = \rk(H/M_S(H)) = \rk(S^\lambda) = \lambda,$$
hence $\Delta_0$ has infinite rank. It follows from \cite[Proposition 2.5.5]{RZ} that $\Delta$ has infinite rank because $\Delta_0$ is open in $\Delta$. By \cite[Corollary 17.1.5]{FJ} this implies that $\rk(\Gamma) = \rk(\Delta) = \rk(\Delta_0) = \lambda$. Since $\Gamma_0$ is open in $\Gamma$ and $\Gamma$ has infinite rank, we find $\rk(\Gamma_0) = \rk(\Gamma)$. Then $\Gamma_0$ is infinite. It follows from Lemma \ref{SRankEqualsRank} that $\rk(\Gamma_0) = r_S(\Gamma_0)$, since $\Gamma_0$ is the image of $\Delta_0 \cong S^\lambda$ under $\phi$. Thus we have shown:
$$r_S(\Gamma_0) = \rk(\Gamma_0) = \rk(\Gamma) = \rk(\Delta) = \rk(\Delta_0) = \lambda.$$
Since $\Gamma_0$ is the image of $G_0$ under $\pi$, we have $r_S(\Gamma_0) \leq r_S(G_0)$. Furthermore, $G_0$ is open in $G$, hence $r_S(G_0)$ is finite by Proposition \ref{SRankSmall}. Then $\lambda = r_S(\Gamma_0) \leq r_S(G_0) < \infty$
contradicts the assumption that $\lambda$ is infinite.
\end{proof}

\begin{Theorem} \label{FrattiniCoverSmall} Every Frattini cover of a small profinite group is small.
\end{Theorem}
\begin{proof} Let $G$ be a small profinite group, and let $\varphi : K \to G$ be a Frattini cover of $G$. By Proposition \ref{SRankSmall} it suffices to show that $r_S(H) < \infty$ for all open subgroups $H$ of $K$ and all finite simple groups $S$, which is the content of Lemma \ref{SRankFinite}
\end{proof}

\section{Elementary equivalence of strongly complete profinite groups}

\begin{Remark} Let $w( \textbf{x})$ be a group word in variables $\textbf{x} = (x_1, ..., x_n)$, see \cite[Section 17.5]{FJ}, and let $G$ be a group. \begin{enumerate}[(i)]

\item We denote by $w(G)$ the subgroup of $G$ generated by the elements of the form $w(\textbf{g})$ with $\textbf g \in G^n$, the \textit{verbal subgroup} corresponding to $w$ in $G$. Note that $w(G)$ is a normal subgroup of $G$.

\item Let $r \in \N$. We write $\length_w(G) \leq r$ if every element of $w(G)$ is of the form $w(\textbf{g}_1)^{\epsilon_1} \cdots w(\textbf{g}_r)^{\epsilon_r}$ with $\textbf{g}_1, ..., \textbf{g}_r \in G^n$ and $\epsilon_1, ..., \epsilon_r \in \{\pm 1\}$.

\end{enumerate}
\end{Remark}

In \cite[Lemma 1.2]{JL}, the following is proved:

\begin{Lemma} \label{VerbalSubgroupFormula} Let $w(x_1, ..., x_d)$ be a group word and $r \in \N$. Let $\varphi(y_1, ..., y_n)$ be a formula in the language of groups. Then there exists a formula $\varphi'(y_1, ..., y_n)$ in the language of groups such that, for all groups $G$ with $\length_w(G) \leq r$ and all $g_1, ..., g_n \in G$, one has
$$G \models \varphi'(g_1, ..., g_n) \iff G/w(G) \models \varphi(g_1 w(G), ..., g_n w(G)).$$
\end{Lemma}

In \cite[Proposition 5.2 (b)]{W}, the following is proved:

\begin{Lemma} \label{VerbalSubgroupLength} Let $G$ be a profinite group and $w$ a group word. Then $w(G)$ is closed in $G$ if and only if there exists $r \in \N$ with $\length_w(G) \leq r$.
\end{Lemma}

Now we can prove that openness of a verbal subgroup is preserved under elementary equivalence. We follow the proofs of Claims A and B in the proof of \cite[Theorem 1.7]{JL}.

\begin{Lemma} \label{EquivalenceVerbalSubgroup} Let $w$ be a group word. Let $G$ and $H$ be profinite groups with $G \equiv H$. If $w(G)$ is open in $G$, then $w(H)$ is open in $H$ and $G/w(G) \cong H/w(H)$.
\end{Lemma}
\begin{proof} Suppose that $w(G)$ is open in $G$. By Lemma \ref{VerbalSubgroupLength} we have $\length_w(G) \leq r$ for some $r \in \N$. We claim that $\length_w(H) \leq r$.

Let $x_1, ..., x_n$ be the variables occurring in $w$. Consider the formula $\varphi(x)$ given by
$$\exists \textbf h_1, ..., \textbf h_r \bigvee_{\epsilon \in \{\pm1\}^r} x = w(\textbf h_1)^{\epsilon_1} \cdots w(\textbf h_r)^{\epsilon_r},$$
where the $\textbf h_j$ are $n$-tuples of variables. Because $\length_w(G) \leq r$, we have $G \models \varphi(g)$ if and only if $g \in w(G)$, for all $g \in G$.

Let $s > r$ be an integer and let $\delta \in \{\pm1\}^s$. By the above, the following sentence holds in $G$, where the $\textbf g_i$ are $n$-tuples of variables:
$$\forall \textbf g_1, ..., \textbf g_s \ \ \varphi(w(\textbf g_1)^{\delta_1} \cdots w(\textbf g_s)^{\delta_s}).$$
Because of $G \equiv H$, this sentence then holds in $H$ as well. Since this is true for all $s > r$ and all $\delta \in \{\pm1\}^s$, it follows that $\length_w(H) \leq r$ and so $w(H)$ is closed by Lemma \ref{VerbalSubgroupLength}.

Since $w(G)$ is open in $G$, the group $G/w(G)$ is finite. Because $\length_w(G) \leq r$ and $\length_w(H) \leq r$, it follows from Lemma \ref{VerbalSubgroupFormula} and $G \equiv H$ that $G/w(G) \equiv H/w(H)$. Since $G/w(G)$ is finite, it follows that $G/w(G) \cong H/w(H)$. Then $w(H)$ has finite index in $H$, hence $w(H)$ is open in $H$.
\end{proof}

Now we can establish a sufficient condition for two elementarily equivalent profinite groups to be isomorphic. The proof of the next statement is conceptually the same as the proof of \cite[Theorem 1.7]{JL}, where the same conclusion is derived in the case of finitely generated profinite groups.

\begin{Proposition} \label{IsomorphismCondition} Let $G$ and $H$ be profinite groups with $G \equiv H$. Suppose that $G$ is small and that for each finite group $A$ there exists a group word $w$ with $w(A) = \{1\}$ such that $w(G)$ is open in $G$. Then $G \cong H$.
\end{Proposition}
\begin{proof} Since $G$ is small, it suffices by \cite[Proposition 16.10.7]{FJ} to show that $G$ and $H$ have the same continuous finite quotients. Let $A$ be a finite group. We claim that $A$ is a continuous quotient of $G$ if and only if $A$ is a continuous quotient of $H$.

Let $w$ be a group word with $w(A) = \{1\}$ such that $w(G)$ is open in $G$. From Lemma \ref{EquivalenceVerbalSubgroup} it follows that $w(H)$ is open in $H$ and $G/w(G) \cong H/w(H)$.

Now we show that $A$ is a continuous quotient of $H$ if and only if it is a quotient of $H/w(H)$. Suppose that $N$ is an open normal subgroup of $H$ such that $H/N \cong A$. We have $w(A) = \{1\}$ by choice of $w$ and so $w(H/N) = \{1\}$, in other words $w(H) \subseteq N$. Then $A \cong H/N \cong \frac {H/w(H)}{N/w(H)}$ is a quotient of $H/w(H)$. Conversely, if $A$ is a quotient of $H/w(H)$, then $A$ is a continuous quotient of $H$, since $w(H)$ is open in $H$. Observe that the same statement is true when replacing $H$ with $G$.

From $G/w(G) \cong H/w(H)$ it follows that $A$ is either a quotient of both $G/w(G)$ and $H/w(H)$ or of neither. Thus, $A$ is a continuous quotient of $G$ if and only if it is a continuous quotient of $H$.
\end{proof}

In \cite{JL}, the existence of such group words as in Proposition \ref{IsomorphismCondition} is derived from \cite[Theorem 1.2]{NS} and the assumption that one of the profinite groups is finitely generated. We derive the existence of such group words for strongly complete profinite groups from Theorem \ref{CharStronglyComplete}.

\begin{Remark} Let $\mathcal V$ be a class of groups. Then $\mathcal V$ is a \textit{variety} if there exists a set $W$ of group words such that a group $G$ lies in $\mathcal V$ if and only if $w(G) = \{1\}$ for all $w \in W$; we say that $\mathcal V$ is the \textit{variety defined by $W$}. The \textit{verbal subgroup of $G$ corresponding to $\mathcal V$} is the subgroup of $G$ generated by $w(G)$ for all $w \in W$. (Note that this depends only on $\mathcal V$, not on $W$.)

Let $G$ be a group, and let $W$ be the set of all group words $w$ with $w(G) = \{1\}$. The \textit{variety generated by $G$} is the variety defined by $W$.
\end{Remark}

The following characterization was proved in \cite[Theorem 2]{SW}.

\begin{Theorem} \label{CharStronglyComplete} Let $G$ be a compact Hausdorff group. The following are equivalent: \begin{enumerate}[(i)]
\item Every subgroup of $G$ of finite index is open.
\item $G$ has only finitely many subgroups of index $n$ for each $n \in \N$.
\item $G$ has only countably many subgroups of finite index.
\item $V$ is open in $G$ whenever $V$ is a verbal subgroup of $G$ corresponding to the variety generated by a finite group.
\end{enumerate}
\end{Theorem}

First note that we can deduce the following:

\begin{Corollary} \label{StronglyCompleteSmall} Every strongly complete profinite group is small.
\end{Corollary}

\begin{Remark} One might ask whether the converse of this statement is true, i.e. if every small profinite group is strongly complete. This question is answered negatively in \cite[Proposition 27]{N} and \cite[Theorem 1]{S}, where counter-examples are constructed.
\end{Remark}

Regarding the fourth equivalent condition, the following theorem of Oates and Powell is of importance to us, see \cite[Corollary 52.12]{Nm}:

\begin{Theorem} \label{FiniteBasis} Let $A$ be a finite group and $\mathcal V$ the variety generated by $A$. Then there exists a finite set of group words $W$ such that $\mathcal V$ is the variety defined by $W$.
\end{Theorem}

Combining Proposition \ref{IsomorphismCondition} with Theorem \ref{CharStronglyComplete} (iv) then yields the following:

\begin{Theorem} \label{EquivalentStronglyComplete} Let $G$ and $H$ be profinite groups with $G \equiv H$. If one of them is strongly complete, then $G \cong H$.
\end{Theorem}
\begin{proof} Suppose that $G$ is strongly complete. Then $G$ is small by Corollary \ref{StronglyCompleteSmall}. We want to apply Proposition \ref{IsomorphismCondition}, hence we have to show that for each finite group $A$ there exists a group word $w$ with $w(A) = \{1\}$ such that $w(G)$ is open in $G$.

Let $A$ be a finite group. Let $\mathcal V$ be the variety generated by $A$. By Theorem \ref{FiniteBasis} there exists a finite set of group words $W$ such that $\mathcal V$ is the variety defined by $W$. Write $W = \{w_1, ..., w_n\}$ and set $w(\textbf x_1, ..., \textbf x_n) := w_1(\textbf x_1) \cdots w_n(\textbf x_n)$, where the $\textbf x_i$ are disjoint tuples of variables of suitable length. Since $\mathcal V$ is the variety generated by $A$, we have $w(A) = \{1\}$ and $w(G)$ is the verbal subgroup of $G$ corresponding to $\mathcal V$. Since $G$ is strongly complete, $w(G)$ is open in $G$ by Theorem \ref{CharStronglyComplete} (iv).
\end{proof}

\begin{Remark} The proof of Theorem \ref{EquivalentStronglyComplete} shows that every strongly complete profinite group satisfies the assumptions of Proposition \ref{IsomorphismCondition}. Conversely, the proof of Proposition \ref{IsomorphismCondition} shows that every profinite group which satisfies these assumptions is strongly complete: if $G$ is such a group and $H$ is a subgroup of finite index, then $N := \bigcap_{g \in G} g^{-1} H g$ is normal in $G$ of finite index and is contained in $H$. Then $G/N$ is a finite group, hence there exists a group word $w$ with $w(G/N) = \{1\}$ such that $w(G)$ is open in $G$. It follows that $w(G) \subseteq N \subseteq H$, hence $H$ is open in $G$.
\end{Remark}

The following statement gives examples of strongly complete groups.

\begin{Proposition} \label{AcceptableStronglyComplete} Let $(G_i)_{i \in I}$ be a family of strongly complete groups such that for every prime number $p$ there are only finitely many $i \in I$ with $p \mid \# G_i$ (where $\# G_i$ denotes the order of $G_i$, see \cite[Section 2.3]{RZ}). Then $G := \prod_{i \in I} G_i$ is strongly complete.
\end{Proposition}
\begin{proof} Let $H$ be a subgroup of $G$ of index $n \in \N$. By assumption, the set $I_n$ of those $i \in I$ with $\gcd(n, \# G_i) \neq 1$ is finite. Let $I_0 := I \setminus I_n$ and put $K_n := \prod_{i \in I_n} G_i$ and $K_0 := \prod_{i \in I_0} G_i$. Let $H_0 := H \cap K_0$ and $H_n := H \cap K_n$.

Since $K_n$ is a product of finitely many strongly complete groups, it is strongly complete, hence $H_n$ is open in $K_n$. By \cite[Proposition 4.2.3]{RZ} we have $(K_0 : H_0) \mid \#K_0$; moreover, $(K_0 : H_0) \mid n$ and $\gcd(n, \# K_0) = 1$, hence $H_0 = K_0$ is open in $K_0$.

It follows that $H_0 \times H_n$ is open in $G$ and contained in $H$, hence $H$ is open in $G$.
\end{proof}

\begin{Example} \begin{enumerate}[(i)] 

\item Small pronilpotent groups are strongly complete. This follows from the fact that pronilpotent groups can be written as the direct product of their Sylow subgroups, see \cite[Proposition 2.3.8]{RZ}. If the group is small, these Sylow subgroups are also small and therefore finitely generated by \cite[Proposition 22.7.10]{FJ}, hence strongly complete and so we can write every small pronilpotent group as in Proposition \ref{AcceptableStronglyComplete}.

\item Consider the group $S_3 \times \prod_p \Z_p^p$, where $p$ ranges over all prime numbers. By Proposition \ref{AcceptableStronglyComplete}, this group is strongly complete, but it is neither pronilpotent, in particular not abelian, nor finitely generated, see \cite[Example 16.10.4]{FJ}.

\end{enumerate}
\end{Example}

One can ask if a result such as Theorem \ref{EquivalentStronglyComplete} can only hold for small, or even strongly complete, profinite groups, and conversely, if it is true for all small groups.

\begin{Question} \label{question} \begin{enumerate} [(i)] \item Let $G$ be a profinite group such that $G \equiv H$ implies $G \cong H$ for all profinite groups $H$. Does it follow that $G$ is small, or strongly complete?

\item Let $G$ be a small profinite group and $H$ a profinite group with $G \equiv H$. Does it follow that $G \cong H$?
\end{enumerate}
\end{Question}

\begin{Remark} Question \ref{question} (i) can be answered positively if $G$ is abelian: Let $G$ be an abelian profinite group which is not small. Consider the complete inverse system $S(G)$ of $G$, see \cite{F}. Since $G$ is not small, there exists an elementary extension $S(G')$ of $S(G)$ whose cardinality is strictly larger than that of $S(G)$. Since $G$ and $G'$ are abelian, $S(G) \equiv S(G')$ implies $G \equiv G'$ by \cite[Proposition 2.20]{F}, but $G$ and $G'$ cannot be isomorphic because $S(G)$ and $S(G')$ are not isomorphic.
\end{Remark}

After this note was finished, Mark Shusterman informed us that he constructed an example which implies a negative answer to Question \ref{question} (i).

\section*{ Acknowledgments }

This work is based on results of the author's master's thesis \cite{H} supervised by Arno Fehm at the University of Konstanz. The author would like to thank Arno Fehm for very helpful discussions on the topic and suggestions concerning this work. The author would also like to thank Moshe Jarden for comments on a previous version.

\end{document}